\documentclass{amsart}

\usepackage{dsfont}       
\usepackage[inline,shortlabels]{enumitem}
\usepackage{csquotes}
\usepackage{amssymb}
\usepackage{thmtools}
\usepackage{acro}
\usepackage{aligned-overset}
\usepackage{mathtools}
\PassOptionsToPackage{hyphens}{url}
\usepackage{hyperref}

\theoremstyle{plain}
\newtheorem{theorem}{Theorem}[section]
\newtheorem{lemma}[theorem]{Lemma}

\theoremstyle{definition}
\newtheorem{definition}[theorem]{Definition}

\theoremstyle{remark}
\newtheorem{remark}[theorem]{Remark}

\numberwithin{equation}{section}

\newcommand{\raisedtarget}[1]{%
  \raisebox{\fontcharht\font`P}[0pt][0pt]{\hypertarget{#1}{}}%
}

\setlist[1]{labelindent=\parindent} 
\setlist[itemize]{leftmargin=*}
\setlist[enumerate,1]{
  label = \alph*),
  ref = \alph*)}

\DeclareAcronym{CDF}{
    short = c.d.f\acdot,
    long = cumulative distribution function
}

\DeclareAcronym{QF}{
    short = q.f\acdot,
    long = quantile function
}

\DeclareAcronym{ae}{
    short = a.e\acdot,
    long = almost everywhere,
    first-style = short
}

\DeclareAcronym{wrt}{
    short = w.r.t\acdot,
    long = with respect to,
    first-style = short
}

\DeclareAcronym{wlog}{
    short = w.l.o.g\acdot,
    long = without loss of generality,
    first-style = short
}

\DeclareAcronym{ie}{
    short = i.e\acdot,
    long = that is,
    first-style = short
}

\DeclareAcronym{iid}{
    short = i.i.d\acdot,
    long = independent and identically distributed,
    first-style = short
}

\DeclarePairedDelimiter{\parenthesis}\lparen\rparen

\newcommand{\NamedFunction}[2]{
    \NewDocumentCommand{#1}{sogd()}{
    \operatorname{#2}
    \IfNoValueTF{##3}           
        {\IfNoValueF{##4}
            {
                \IfBooleanTF{##1}
                {\parenthesis*}%
                {\parenthesis[##2]}
                {##4}%
            }%
        }%
        {
            \IfBooleanTF{##1}%
            {\parenthesis*}%
            {\parenthesis[##2]}{##3}%
            \IfNoValueF{##4}{(##4)}
        }%
    }%
}

\NewDocumentCommand{\DefineDelimiterBefore}{mm}{
  \NewDocumentCommand{#1}{somm}
    {##3\IfBooleanTF{##1}{#2*}{#2[##2]}{##4}}  

}

\newcommand{\NN}{\mathbb{N}}

\newcommand{\RR}{\mathbb{R}}

\newcommand{\BC}{\mathcal{B}}

\DeclareMathOperator{\LeftRightarrow}{\Leftrightarrow}          

\providecommand\given{}
\newcommand\SetSymbol[1][]{%
  \nonscript\:#1\vert
  \allowbreak
  \nonscript\:
  \mathopen{}}

\DeclarePairedDelimiterX\set[1]\{\}{%
\renewcommand\given{\SetSymbol[\delimsize]}
#1
}

\NamedFunction{\setcount}{\#}

\newcommand{\defeq}{\coloneqq}
\newcommand{\eqdef}{\eqqcolon}

\newcommand{\R}{\RR}                         
\newcommand{\Rbar}{\overline{\R}}                               
\newcommand{\N}{\NN}                         

\DeclareMathOperator{\Id}{id}
\newcommand{\id}{\Id}

\DefineDelimiterBefore{\Functionimage}{\parenthesis}
\NewDocumentCommand{\preimage}{sogg}{\IfBooleanTF{#1}{\Functionimage*}{\Functionimage[#2]}{{#3}^{-1}}{#4}}

\DeclarePairedDelimiterX{\restrBase}[1]{.}{\rvert}{
  \kern-\nulldelimiterspace %
  #1 
  \littletaller 
  }
  
\newcommand{\littletaller}{\mathchoice{\vphantom{\big|}}{}{}{}}

\NewDocumentCommand{\restr}{somt_m}{
    \IfBooleanTF{#1}
                {\restrBase[#2]}%
                {\restrBase*}
    {#3}_{#5}
}

\newcommand{\convexindex}{w}

\newcommand{\Leftgeneralizedinverse}[1]{#1^-}   
\newcommand{\Rightgeneralizedinverse}[1]{#1^+}   

\NewDocumentCommand{\interior}{som}{\IfBooleanTF{#1}{\parenthesis*}{\parenthesis[#2]}{#3}^\circ}
\newcommand{\closure}[1]{\overline{#1}}

\newcommand{\RadonNikodym}[2]{\frac{\diff#1}{\diff#2}}

\newcommand*\diff{\mathop{}\!\mathrm{d}}

\NamedFunction{\B}{\BC}                        

\newcommand{\Lebesgue}{\lambda}

\NamedFunction{\support}{spt}

\newcommand{\pushforward}[2]{#1_\##2}
\newcommand{\ind}{\mathds{1}}               
\newcommand{\1}{\ind}

\newcommand{\uniformmeasure}{\restr{\Lebesgue}{(0,1)}}
\NamedFunction{\betad}{Beta}
\newcommand{\Beta}{\betad}

\newcommand{\setletter}{A}

\begin{document}
\title{Quantile characterization of univariate unimodality}
\author{Markus Zobel}
\address{Institute for Mathematical Stochastics, University of Göttingen}
\email{markus.zobel@uni-goettingen.de}
\author{Axel Munk}
\address{Institute for Mathematical Stochastics, University of Göttingen}
\email{munk@math.uni-goettingen.de}

\date{11.02.2026}
\begin{abstract}
    Unimodal univariate distributions can be characterized as piecewise convex-concave \aclp*{CDF}.
    In this note we transfer this shape constraint characterization to the \acl*{QF}.
    We show that this characterization comes with the upside that the \acl*{QF} of a unimodal distribution is always absolutely continuous and consequently unimodality is equivalent to the quasi-convexity of its Radon-Nikodym derivative, \ac{ie}, the quantile density. Our analysis is based on the theory of generalized inverses of non-decreasing functions and relies on a version of the inverse function rule for non-decreasing functions.
\end{abstract}
\maketitle
\section{Introduction}
\label{sec: Introduction}
\subsection{Characterization of unimodality}
We will start with the intuitive characterization of a univariate unimodal (probability) distribution through its (Lebesgue-) density (if it exists): The density is first non-decreasing until a mode and then non-increasing (left of \autoref{fig: unimodal versions}).
This description matches directly with the density being quasi-concave.
\begin{definition}
    \begin{enumerate}[wide, labelwidth=!, labelindent=0pt,itemindent=!]
        \item A function $f:I\to\Rbar$ defined on an open subinterval $I\subseteq\Rbar$ is \emph{quasi-concave} if for all $x,y\in I$ and $\convexindex\in(0,1)$ it holds
              \begin{equation}
                  f(\convexindex x+ (1-\convexindex) y)\geq\min(f(x),f(y)).
                  \label{def: quasi-concaity}
              \end{equation}
        \item A function $f:I\to\Rbar$ is \emph{quasi-convex} if its negative $-f$ is quasi-concave, \ac{ie}, if it satisfies for all $x,y\in I$ and $\convexindex\in(0,1)$ the inequality
              \begin{equation*}
                  f(\convexindex x+ (1-\convexindex) y)\leq\max(f(x),f(y)).
              \end{equation*}
    \end{enumerate}
\end{definition}
Indeed, this definition is equivalent to the qualitative description above:
A function $f:I\to\Rbar$ is quasi-concave (-convex) if and only if there exists a disjoint decomposition of $I$ into two intervals $A,B\subseteq I=A\sqcup B$, where for all $x\in A,y\in B$ we have $x<y$, such that $f$ is non-decreasing (non-increasing) on $A$ and non-increasing (non-decreasing) on $B$ \cite[Corollary 20]{martosNonlinearProgrammingTheory1975}.
These intervals can be open, closed or half-open and one of $A,B$ can even be empty.
Each such a decomposition defines a \emph{mode} of a quasi-concave (-convex) function as $\nu\defeq\sup(A)=\inf(B),\nu\in\Rbar$.\footnote{In the case where one of $A,B$ is empty the definition based on the other can be taken.}
Intuitively, $\nu$ is the point at which the monotonicity switches and a continuous quasi-concave $f$ is maximized at $\nu$.
In the case where a quasi-concave $f$ has multiple local maxima the decomposition of the domain is not unique and all potential modes are collected as the \emph{modal interval} $[\nu_{\min},\nu_{\max}]$ which is always closed.
\begin{definition}
    \label{def:density unimodal}
    An absolutely continuous distribution $\mu:\B(\R)\to[0,1]$ is called \emph{dens-unimodal} if it has a quasi-concave density $f_{\mu}$ (\ac{wrt} to the Lebesgue measure $\Lebesgue$ on $(\R,\B(\R))$). Any mode of $f_\mu$ is also called a \emph{mode} of $\mu$.
\end{definition}
This definition can be generalized beyond the absolute continuous case based solely on properties of the \ac{CDF} $F_\mu:\R\to[0,1]$ of $\mu$ defined as $F_\mu(x)\defeq\mu(-\infty,x]$ for all $x\in\R$.
\begin{definition}[{\cite[Definition 1.1.]{Dharmadhikari.1988}}]
    \label{def: Unimodality CDF}
    A univariate distribution $\mu$ is \emph{\ac{CDF}-unimodal} if there exists a mode $\nu\in\R$ such that the \ac{CDF} $F_\mu$ is convex on $(-\infty,\nu]$ and concave on $[\nu,+\infty)$.
\end{definition}
\begin{figure}
    \centering
    \includegraphics{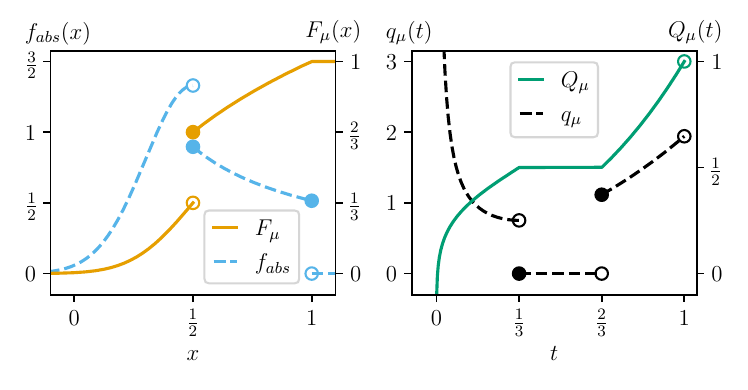}
    \caption{The distribution $\mu$ is \ac{CDF}-unimodal: The density $f_{abs}$ (left, dashed) of $\mu_{abs}$, the absolute continuous part of $\mu$, is quasi-concave (non-decreasing on $A=(-\infty,\frac12)$ and non-increasing on $B=[\frac12,+\infty)$), hence, $\mu_{abs}$ is dens-unimodal. The \ac{CDF} $F_\mu$ (left, filled) is convex on $(-\infty,\frac12]$ and concave on $[\frac12,+\infty)$. The quantile density $q_\mu$ (right, dashed) is quasi-convex with modal interval $[\alpha_{\min},\alpha_{\max}]=[\frac{1}{3},\frac{2}{3}]$ (non-increasing on $(0,\alpha)$ and non-decreasing on $[\alpha,1)$ for all $\alpha\in[\alpha_{\min},\alpha_{\max}]$). The \acs{QF} $Q_\mu$ (right, filled) is concave on $(0,\alpha]$ and convex on $[\alpha,1)$, for $\alpha\in[\alpha_{\min},\alpha_{\max}]$. Empty circles denote left or right limits, while the filled circles show the values attained by the respective functions at their discontinuities. The choices for the densities were made arbitrary.}
    \label{fig: unimodal versions}
\end{figure}
\begin{remark}\begin{enumerate}[wide, labelwidth=!, labelindent=0pt,itemindent=!]
        \item \autoref{def: Unimodality CDF} is a generalization of \autoref{def:density unimodal}: Any \ac{CDF}-unimodal distribution $\mu$ is a mixture of an absolutely continuous dens-unimodal distribution and a Dirac measure at a mode of the latter (see, \autoref{lem: Relation dens unimodal and cdf unimodal}). Every mode of $\mu$ \ac{wrt} \ac{CDF}-unimodality is also a mode \ac{wrt} to dens-unimodality and vice versa.
        \item What we term \ac{CDF}-unimodality corresponds to unimodality in \cite{Dharmadhikari.1988}.
              The set of all \ac{CDF}-unimodal distribution is closed \ac{wrt} the topology of weak convergence of distributions \cite[Theorem 1.1.]{Dharmadhikari.1988} and there is a Choquet-type representation for all \ac{CDF}-unimodal distributions with mode $0$ \cite[Theorem 1.2.]{Dharmadhikari.1988}.
    \end{enumerate}
\end{remark}
\subsection{Main results}
In this paper we provide a new equivalent notion of \ac{CDF}-unimodality.
The result came about as a useful framing while studying properties of Wasserstein-2 barycenters, but it stands on it own in its simplicity.
The starting point for the characterizations of \ac{CDF}-unimodality we propose is the fact that the \ac{QF} $Q_\mu:(0,1)\to\R$, defined as
\begin{equation}
    \label{eqn: Definition quantile function}
    Q_{\mu}(t)\defeq\inf(\set{x\in\R\given t\leq F_\mu(x)}),\quad\forall t\in(0,1),
\end{equation}
of a \ac{CDF}-unimodal distribution enjoys higher regularity.
\begin{theorem}
    \label{thm:unimodaldistributionhasabsolutelycontinuousqunatilefunction}
    The \ac{QF} $Q_{\mu}$ of a \ac{CDF}-unimodal distribution is absolutely continuous.
\end{theorem}
A proof is given in \autoref{sec: Proof of absolute continuity o fgeneralized inverse of unimodal genearting}.
For now, we define \emph{absolute continuity of a non-decreasing function}, like the \ac{QF} $Q_\mu:(0,1)\to\R$, as the existence of a \emph{Radon-Nikodym derivative} $q_\mu:(0,1)\to\Rbar_{\geq 0}$ such that $Q_\mu(y)-Q_{\mu}(x)=\int_{x}^yq_{\mu}\diff\Lebesgue$ for all $x,y\in(0,1)$. The Radon-Nikodym derivative $q_{\mu}$ of a \ac{QF} $Q_\mu$ is called the \emph{quantile density} of $\mu$ \cite{parzenNonparametricStatisticalData1979}.
The following result is the main result of this paper and provides a characterization of \ac{CDF}-unimodal distributions $\mu$ based solely on properties of their quantile densities $q_{\mu}$, similar to \autoref{def:density unimodal}.
\begin{theorem}
    \label{thm: Unimodal distributions via quantile}
    A univariate distribution $\mu$ is \ac{CDF}-unimodal if and only if
    \begin{enumerate}[left=\parindent]
        \item
              \label{enum: quantile density definition of unimodality} there exists a quantile density $q_\mu$ of $\mu$ which is quasi-convex, or
        \item
              \label{enum: Qunatife function definition of unimodality} there exists a \emph{quantile mode} $\alpha\in[0,1]$ such that the \ac{QF} $Q_\mu$ is concave on $(0,\alpha]$ and convex on $[\alpha,1)$.
    \end{enumerate}
\end{theorem}
A proof is given in \autoref{sec: Proof of characterization of Unimodality via quantile functions}.
Intuitively, \ref{enum: Qunatife function definition of unimodality} is equivalent to \autoref{def: Unimodality CDF} as the inverse of a non-decreasing convex function is non-decreasing and concave, and vice versa.
Our proof of \autoref{thm: Unimodal distributions via quantile} will not use this approach directly but instead focuses on the relation of the quantile density $q_\mu$ and the density $f_{abs}$ of the absolutely continuous dens\nobreakdash-unimodal part $\mu_{abs}$ of a general \ac{CDF}-unimodal distribution $\mu$. The two are related by a generalized inverse function rule for Radon-Nikodym derivatives
\begin{equation}
    f_{abs}=\frac{1}{q_{\mu}\circ F_\mu}\quad\text{$\Lebesgue$-\ac{ae} on $\support{\mu}$,}
    \label{eq: initial inverse function rule}
\end{equation}
\ac{ie}, there is a null set $N\in\B(\R)$ such that $\set{x\in\support{\mu}\given f_{abs}(x)\neq\frac{1}{q_{\mu}\circ F_\mu}(x)}\subseteq N$ and $\Lebesgue(N)=0$. Here, $\support{\mu}$ denotes the support of $\mu$, \ac{ie}, the smallest closed subset $A\subseteq\R$ such that $\mu(\R\setminus A)=0$.
\subsection{Novelty}
To the best of our knowledge \autoref{thm:unimodaldistributionhasabsolutelycontinuousqunatilefunction} and \autoref{thm: Unimodal distributions via quantile} are new.
The characterization \ref{enum: Qunatife function definition of unimodality} of \autoref{thm: Unimodal distributions via quantile} appears in \cite[Section 6.1.]{backhoffStochasticGradientDescent2025} for absolutely continuous $\mu$ with a subtle difference. They only require the \ac{QF} to be concave and convex on the open intervals $(0,\alpha),(\alpha,1)$ respectively. This condition is only necessary but not sufficient for \ac{CDF}-unimodality.
These requirements do not guarantee the continuity of the \ac{QF} at the quantile mode. A discontinuity at a quantile mode $\alpha\in(0,1)$ corresponds to a flat region in the \ac{CDF}. We then can decompose the \ac{CDF} into four regions such that the function is  convex, concave, convex, concave. Therefore, the corresponding \ac{CDF} is \ac{CDF}-bimodal. A counter example is the (absolutely continuous) mixture of two uniform distributions. Take $\mu=\restr{\Lebesgue}{(0,\frac12)}+\restr{\Lebesgue}{(\frac32,2)}$ then its \ac{QF} $Q_\mu=\restr{\Id}{(0,1)} +\1_{(\frac12,1)}$ is \emph{affine}, \ac{ie}, convex and concave, on both $(0,\frac12)$ and $(\frac12,1)$. But, the measure $\mu$ is clearly dens-bimodal (non-decreasing, non-increasing, non-decreasing, non-increasing density) and thus not \ac{CDF}-unimodal.

\subsection{Implications}
\label{sec: Implications}
The authors of \cite{backhoffStochasticGradientDescent2025} use their characterization to investigate the (dens-) unimodality of the Wasserstein-2 barycenter of (dens-) unimodal distribution, as the \ac{QF} of the Wasserstein-2 barycenter of a distribution of distributions is the point-wise average of the random \acp{QF}{}.
Therefore, characterizing distributions via their \ac{QF} is pertinent to prove the preservation of these properties when taking the Wasserstein-2 barycenter.
While this idea was also the starting point of this work, we want to emphasize the simplicity of the new characterization of \ac{CDF}-unimodality of distributions via their \ac{QF} in \autoref{thm: Unimodal distributions via quantile}.
It encompasses the standard definition of unimodality \cite{Dharmadhikari.1988} and reveals that the quantile perspective on unimodality can be used equally well as we always have enough regularity for the existence of a quantile density.
\subsection{Organization of the paper}
The groundwork for inferring properties of the \ac{QF} from properties of the \ac{CDF}, and vice versa will be laid in \autoref{sec: Non-Decreasing Functions}. As every \ac{CDF} is non-decreasing and the \ac{QF} is its left-continuous non-decreasing generalized inverse, we will focus on non-decreasing functions, their generalized inverses and associated measures there.
In \autoref{sec: Abosultely continuity of the generalized inverse} we collect and adapt results from the literature to prove the aforementioned inverse function rule \eqref{eq: initial inverse function rule} for non-decreasing real-valued functions and their generalized inverse (\autoref{thm: Radon-Nikodym derivative generlized inverse function}).
Then, in \autoref{sec: Unimodality of locally finite measures} we apply these results to prove \autoref{thm:unimodaldistributionhasabsolutelycontinuousqunatilefunction} and \autoref{thm: Unimodal distributions via quantile}. We conclude by extending our results to locally finite Borel measures.

\section{Non-decreasing functions}
\label{sec: Non-Decreasing Functions}
The theory of non-decreasing functions and its connection to measure theory is standard \cite{natansonTheorieFunktionenReellen1961,hewittRealAbstractAnalysis1969}.
We will start with the general framework of extended real-valued non-decreasing functions.
Their generalized inverses are again non-decreasing functions whose graphs are the graphs of the original functions with the axes flipped (\autoref{fig: generalized inverse}).
We introduce the correspondence between real-valued non-decreasing functions on open intervals and Borel measures there. In this work the \emph{extended real numbers} are $\Rbar\defeq\R\cup\set{-\infty,+\infty}$ with the usual topology and $-\infty< x< +\infty$ for all $x\in\R$. We formally set $\pm\infty + x= \pm\infty$ for all $x\in\R$.
\subsection{Non-decreasing functions and generalized inverses}
It is well-known \cite[Theorem 4.29]{rudinPrinciplesMathematicalAnalysis1976} that any \emph{non-decreasing} function $G:\R\to\Rbar$, \ac{ie}, for all $x\leq y\in\R$ it holds $G(x)\leq G(y)$,
has left- and right limits at any $x\in\R$, denoted as
\begin{equation*}
    \lim_{y\nearrow x}G(y)\eqdef G(x-) ,\quad\text{respectively} \lim_{y\searrow x}G(y)\eqdef G(x+).
\end{equation*}
Furthermore, $G$ is continuous except for countably many jump discontinuities \cite[(8.19)]{hewittRealAbstractAnalysis1969}. The non-decreasingness enables us to obtain for all $x\in\R$\begin{equation*}
    G(x-)=\sup_{y\in(-\infty,x)}G(y)\leq G(x)\leq\inf_{y\in(x,+\infty)}G(y)=G(x+).
\end{equation*} This allows us to define its \emph{left-continuous} or \emph{lower-semicontinuous}  version and \emph{right-continuous} or \emph{upper-semicontinuous} version
\begin{equation*}
    \begin{aligned}
        G_l(x) & \defeq G(x-)=\sup_{y\in(-\infty,x)}G(y),
        \\
        G_r(x) & \defeq G(x+)=\inf_{y\in(x,+\infty)}G(y),
    \end{aligned}
    \quad\forall x\in\R.
\end{equation*}
We directly have that $G_l\leq G\leq G_r$. Any other non-decreasing function $F$ satisfying $G_l\leq F\leq G_r$ is thus $\Lebesgue$-\ac{ae} equal to $G$ and may only differ from it at the shared discontinuities and $F_l=G_l,F_r=G_r$. We call any such non-decreasing function $F$ a \emph{version} of $G$. A non-decreasing function $G:\R\to\Rbar$ together with all its versions ensembles to a $\Lebesgue$-\ac{ae} equivalence class restricted to non-decreasing functions denoted as $[G]$.
\begin{remark}
    The \ac{CDF} $F_\mu:\R\to[0,1]$ is an example of a non-decreasing real-valued function. By the notation from above $F_\mu=(F_\mu)_r\in[F_\mu]$ is the right-continuous version of its equivalence class of non-decreasing functions.
\end{remark}
For a non-decreasing function $G:\R\to\Rbar$ a \emph{generalized inverse} $H:\R\to\Rbar$ is a member of a class $[H]$ of non-decreasing functions specified by either its left- or right-continuous version defined as
\begin{equation}
    \label{eqn: definition generalized inverse}
    \begin{aligned}
        H_l(t) & \defeq \Leftgeneralizedinverse{G}(t)\defeq\inf(\set{x\in\R\given t\leq G(x)}),  \\
        H_r(t) & \defeq \Rightgeneralizedinverse{G}(t)\defeq\sup(\set{x\in\R\given t\geq G(x)}),
    \end{aligned}
    \quad\forall t\in\R,
\end{equation}
\ac{ie}, $H_l\leq H\leq H_r$ and $H$ is non-decreasing \cite{fengNoteGeneralizedInverses2012,embrechtsNoteGeneralizedInverses,fortelleGeneralizedInversesIncreasing}. It is well known that $[H]$ is well-defined \cite[Section 3]{delafortelleStudyGeneralizedInverses2015}, \ac{ie}, $\Leftgeneralizedinverse{G}$ is indeed non-decreasing and left-continuous, $\Rightgeneralizedinverse{G}$ is non-decreasing and right-continuous, and they agree $\Lebesgue$-\ac{ae}.
Therefore, all generalized inverses of $G$ are versions of each other and make up the equivalence class $[H]$. Furthermore, we have that
\begin{equation*}
    \inf(\set{x\in\R\given t\leq G(x)})=\min(\set{x\in\R\given t\leq G_r(x)}),\quad\forall t\in\R
\end{equation*} by the upper-semicontinuity of $G_r$ \cite[Lemma 2.1]{kampkeIncomeModelingBalancing2015}\footnote{We need to formally set $\min({\emptyset})=+\infty$.} and thus the definition of $[H]$ is independent of the version of $[G]$ used. Additionally, $\Rightgeneralizedinverse{H_r}=G_r$ \cite[Proposition 4.2]{delafortelleStudyGeneralizedInverses2015}, and consequently we can relate $[G]$ and $[H]$ by one being the inverse class of the other, \ac{ie}, to be a generalized inverse is a class property (\autoref{fig: generalized inverse}).
At the core of the relation between the two classes is a fundamental equivalence \cite[Lemma 1 (h)(3),(j)(3)]{wackerPleaseNotAnother2023}. It is given as either
\begin{subequations}
    \label{eqn: Galois connection}
    \begin{align}
        G_l(x) > t                  & \LeftRightarrow x > H_r(t),\quad\forall x,t\in\R,\label{eqn: Galois connection greater}      \\
        \text{or}\quad G_l(x)\leq t & \LeftRightarrow x\leq H_r(t),\quad\forall x,t\in\R.\label{eqn: Galois connection less equal}
    \end{align}
\end{subequations}
From now on $G$ and $H$ will always be a pair of generalized inverses, \ac{ie},  generalized inverses of one another, and we will use arbitrary versions instead of emphasizing the underling equivalence sets. In the case where $G$ is bijective we have $[H]=\set{G^{-1}}$ in all other cases $G^{-1}$ denotes the set-valued preimage of $G$.

\begin{figure}
    \centering
    \includegraphics{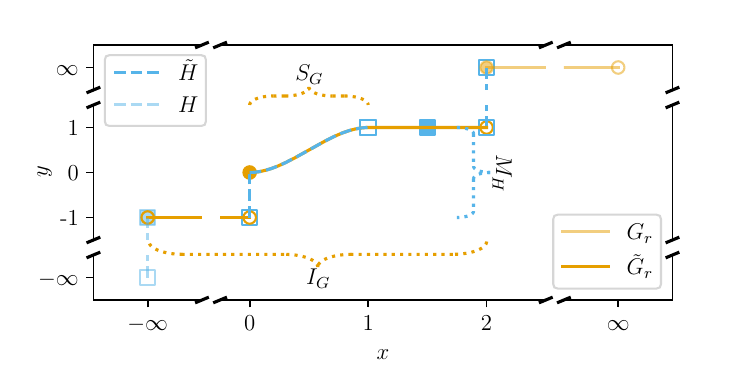}
    \caption{The right-continuous version $G_r$ (orange, filled, transparent) of a non-decreasing function with its real-valued restriction $\Tilde{G}_r:I_G\to\R$ (orange, filled) are plotted as functions mapping $x$ values to  $y$ values while a generalized inverse $H=\frac{H_l+H_r}{2}$ (blue, dashed, transparent) of $G$ and its real restriction $\Tilde{H}:I_H\to\R$ (blue, dashed) are plotted as functions mapping $y$ values to $x$ values. The left and right limits of $G_r$, $H$ are shown as empty circles, respectively empty squares. The filled variants represent the value attained by the function at a discontinuity. The braces visualize the intervals: $S_G=\closure{\preimage{G_r}{M_H}}$, $M_H=\interior{\preimage{H}{I_G}}$, $I_G=\interior{\preimage{G_r}{\R}}$. We have that $\mu_G=\delta_0+\Beta(2,2)$ restricted to $(-\infty,2)=I_G$.}
    \label{fig: generalized inverse}
\end{figure}
\subsection{Non-decreasing real-valued functions}
The class of non-decreasing functions includes non-decreasing real-valued functions $\Tilde{G}:I\to\R$ only defined on an open subinterval $I\subseteq\R$. We embed $\Tilde{G}$ as $G:\R\to\Rbar$ by extending it with $-\infty$ to the left of $I$ and by $+\infty$ to the right of $I$, precisely $G(x)\defeq-\infty$ for all $x\leq\inf(I)$ and $G(x)\defeq+\infty$ for all $x\geq\sup(I)$ and $G(x)=\Tilde{G}$ else.
Conversely, every class of non-decreasing extended real-valued functions $[G]$ can be reduced to a class of real-valued non-decreasing functions on an open subinterval $I_G\subseteq\R$ denoted as the \emph{regular domain},\footnote{At this point we exclude some peculiar cases: \begin{enumerate*}
        \item The constant $-\infty$ and constant $+\infty$ function which would both be represented by the empty function $G:\emptyset\to\R$.
        \item The classes of non-decreasing functions with a jump from $-\infty$ to $+\infty$ at some point $x\in\R$
              (again $I_G$ would be the empty set as $H\equiv x$).
    \end{enumerate*} Combined the excluded functions correspond to all possible constant generalized inverses. Consequently, we also exclude everywhere constant functions $G$.} where
\begin{equation*}
    I_G\defeq(H(-\infty),H(+\infty))\defeq (H(-\infty+),H(+\infty-))\defeq\big(\inf_{y\in\R} H(y),\sup_{y\in\R} H(y)\big).
\end{equation*}
Then $\Tilde{G}\defeq\restr{G}{I_G}:I_G\to\R$ the restriction of $G$ to $I_G$ is a real-valued function, as for every $x\in I_G$ there exists $u,s\in\R$ such that $H(u)<x<H(s)$ and therefore
\begin{equation*}
    -\infty<u\overset{\eqref{eqn: Galois connection less equal}}{\leq} G_r(H(u))\leq G(x)=\Tilde{G}(x)\leq G_l(H(s))\overset{\eqref{eqn: Galois connection less equal}}{\leq} s<+\infty.
\end{equation*}
This hints at the alternative representation $I_G=\interior{\preimage{G}{\R}}$ (\autoref{fig: generalized inverse}). Of which only the inclusion $I_G\supset\interior{\preimage{G}{\R}}$ remains to be shown.
For each $x\in\interior{\preimage{G}{\R}}$ there exist $u<x<s\in\interior{\preimage{G}{\R}}$ such that
\begin{equation*}
    H(-\infty)\leq H_l(G(u))\overset{\eqref{eqn: Galois connection less equal}}{\leq} u<x<s\overset{\eqref{eqn: Galois connection less equal}}{\leq} H_r(G(s))<H(+\infty).
\end{equation*}
\begin{remark}
    The \ac{QF} $Q_\mu:(0,1)\to\R$ is a generalized inverse of $F_\mu$ restricted to its regular domain $I_{Q_\mu}=(F_\mu(-\infty),F_\mu(+\infty))=(0,1)$. By comparing \eqref{eqn: Definition quantile function} and \eqref{eqn: definition generalized inverse} we conclude that $Q_\mu=\big(\Tilde{Q}_\mu\big)_l:(0,1)\to\R$.
\end{remark}
In the following $\Tilde{G}$ will always denote this restriction of a non-decreasing extended real-valued function $G:\R\to\Rbar$ to it regular domain $I_G$.
\subsection{Non-decreasing functions and locally-finite borel measures}
To every non-decreasing extended real-valued function $G:\R\to\Rbar$ we associate a Borel measure $\mu_G$ on $I_G$, \ac{ie}, $\mu_G:\B(I_G)\to\Rbar_{\geq0}$, by fixing the value on open intervals $(x,y)\subseteq I_G$ as
\begin{equation}
    \mu_G((x,y))\defeq G(y-)-G(x+)=G_l(y)-G_r(x).
    \label{eq: measure from monotone function}
\end{equation}
This definition is invariant under constant shifts, \ac{ie}, $\mu_G=\mu_{G+c}$ for each $c\in\R$.
We have that $\mu_G$ is \emph{locally finite}, \ac{ie}, for all compact sets $C\subset I_G$ it holds $\mu_G(C)<+\infty$, as for each compact set $C\subset I_G$ there exists $[x,y]\subset I_G$ with $C\subseteq[x,y]$ and consequently
\begin{equation*}
    \mu_G(C)\leq\mu_G([x,y])={G}_r(x)-{G}_l(y)=\Tilde{G}_r(x)-\Tilde{G}_l(y)<+\infty.
\end{equation*}

Conversely, we can define for a locally finite Borel measure $\mu:\B(I)\to\Rbar_{\geq0}$ on an open interval $I$ a class of non-decreasing functions $[G_\mu]$ by defining a right-continuous version of $\Tilde{G}_\mu:I\to\R$ via
\begin{equation}
    \Tilde{G}_\mu\big(x)=\begin{cases}
        \mu((z,x])  & x>z     \\
        -\mu((x,z]) & x\leq z
    \end{cases},\quad\forall x\in I
    \label{eq: monotone function from measure}
\end{equation}
based on a choice of $z\in I=I_{G_\mu}$ \cite[(19.45)]{hewittRealAbstractAnalysis1969}.
\begin{remark}
    This definition differs from the classical definition of the \ac{CDF}. In the latter one chooses $z=-\infty\in\Rbar=\closure{I_{F_\mu}}$, so the closure of the regular domain of $F_\mu$, instead of its interior as in \eqref{eq: monotone function from measure}. For general locally finite Borel measures $\mu:I\to\Rbar_{\geq0}$ this is not possible as $G_\mu((\inf I_{G_{\mu}})-)=-\infty$ may hold for any distribution function $G_\mu$ of $\mu$.
\end{remark}
Combining \eqref{eq: measure from monotone function} and \eqref{eq: monotone function from measure} gives us a bijection between locally-finite Borel measures  $\mu:\B(\R)\to\Rbar_{\geq0}$ on $\R$ and classes of non-decreasing real-valued functions $[G]$  which satisfy $G_r(0)=0$ with $z=0\in\R$ \cite[(19.48)]{hewittRealAbstractAnalysis1969}. Enriching these classes by adding constant shifts consequently relates all non-decreasing real-valued functions with locally finite Borel measures. This equivalence extends directly to locally finite Borel measures on arbitrary open subsets $I\subseteq\R$ and non-decreasing real-valued functions on $I$. Here, we have to associate each possible $I$ with a fixed choice of $z_I\in I$. Consequently, we can call any non-decreasing function $G:I_G\to\Rbar$ a \emph{distribution function} of $\mu_G$ its \emph{associated measure}. The associated measure $\mu_H:\B(I_H)\to\Rbar_{\geq0}$ of a generalized inverse $H$ of $G$ is called the \emph{inverse measure} of $\mu_G$ \cite[Appendix A]{bobkovOnedimensionalEmpiricalMeasures2019}.
\begin{remark}
    It would be natural to identify all measures $\mu:\B(I)\to\Rbar_{\geq0}$ with one another if when extended by $0$ to $\B(\R)$ they are the same measure. The corresponding distribution functions are very similar: They all are restrictions of the distribution function $G:\R\to\Rbar$ which corresponds to the maximally extended measure $\mu:\B(\R)\to\Rbar_{\geq0}$. This works well for finite measures, but in the case where $\mu(I)=+\infty$ any extension of $\mu$ would no longer be locally finite and thus outside the scope of our analysis. Even in the finite case we run into problems as the inverse measures $\mu_H$ corresponding to generalized inverses of differently restricted distribution functions differ.
\end{remark}
\subsection{Pushforward measures under the generalized inverse}
By the definition of $\mu_G$ we have that $\mu_G(I_G)=\sup_{x\in I_G}G(x)-\inf_{x\in I_G}G(x)=\Lebesgue(M_H)$, where we similarly to the regular domain define the \emph{mass interval} $M_H$ as
\begin{equation*}
    M_H\defeq\big(\inf_{y\in I_G}G(y),\sup_{y\in I_G}G(y)\big)\subseteq I_H,
\end{equation*}
which is consequently again given as $M_H=\interior{\preimage{H}{I_G}}$ (\autoref{fig: generalized inverse}).
The restriction of the Lebesgue measure to $M_H$ denoted as $\restr{\Lebesgue}{M_H}$ contains the same mass as $\mu_G$.
As we eventually want to compare $\restr{\Lebesgue}{M_H}$ with $\mu_H$, we define the extension of $\restr{\Lebesgue}{M_H}$ to $\B(I_H)\supset\B(M_H)$. We denote the extension as $\Lebesgue_H:\B(I_H)\to\Rbar_{\geq0}$, and it satisfies $\Lebesgue_H(I_H\setminus M_H)\defeq0$, $\Lebesgue_H(B)\defeq\restr{\Lebesgue}{M_H}(B)$ for all $B\in\B(M_H)$. This definition connects the generalized inverse $H$ with the measure $\mu_G$ directly.
\begin{lemma}
    \label{lem:pushforward of uniform with generalized inverse}
    For a generalized inverse pair $G,H$ holds both
    \begin{equation}
        \label{eqn: Qunatile generates measure}
        \pushforward{\Tilde{H}}{{\Lebesgue}_{H}}=\mu_G\quad\text{and}\quad\pushforward{\Tilde{G}}{\Lebesgue_G}=\mu_H.
    \end{equation}
\end{lemma}
Here $\pushforward{\Tilde{H}}{\Lebesgue_H}$ denotes the pushforward measure of $\Lebesgue_H$ by $\Tilde{H}$ which is defined as $\pushforward{\Tilde{H}}{\Lebesgue_H}:\B(I_G)\to\Rbar_{\geq0}$ with
\begin{equation*}
    \pushforward{\Tilde{H}}{\Lebesgue_H}(A)
    \defeq\Lebesgue_H(\preimage{\Tilde{H}}{A})
    =\Lebesgue_H(\set{t\in I_H\given \Tilde{H}(t)\in A}),\quad\forall A\in\B(I_G).
\end{equation*}
Further $\Lebesgue_G:\B(I_G)\to\Rbar_{\geq0}$ is analogously the extension of $\restr{\Lebesgue}{M_G}$.
\begin{proof}
    To prove that $\mu_G=\pushforward{\Tilde{H}}{\Lebesgue_H}$ we show that the distribution function $G_r$ of $\mu_G$ and $G_{\pushforward{\Tilde{H}}{\Lebesgue_H}}$ defined by \eqref{eq: monotone function from measure} for an arbitrary $z\in I_G$ are shifted versions of one another, \ac{ie}, $G_{\pushforward{\Tilde{H}}{\Lebesgue_H}}=G_r+c(z)$, with the constant $c(z)$ depending only on the choice of $z$. If the difference of two distribution functions is \ac{ae} constant then the associated measures coincide. Therefore, if that were the case, we would conclude that
    \begin{equation*}
        \mu_{G_r}=\mu_{G_r+c(z)}=\mu_{G_{\pushforward{\Tilde{H}}{\Lebesgue_H}}}=\pushforward{\Tilde{H}}{\Lebesgue_H}.
    \end{equation*}
    Thus, fix $z\in I_G$ and take an arbitrary $x\in I_G$. If $x>z$ it holds
    \begin{align*}
        G_{\pushforward{\Tilde{H}}{\Lebesgue_H}}(x)
        \overset{\eqref{eq: monotone function from measure}} & = \Lebesgue(\set{t\in I_H\given \Tilde{H}(t)\in(z,x]})
        = \Lebesgue(\set{t\in \preimage{\Tilde{H}}{I_G}\given z<\Tilde{H}_l(t)\leq x})  +0                                                                                                                                \\
                                                             & {=} \Lebesgue(\set{t\in M_H\given G_r(z)\overset{\eqref{eqn: Galois connection greater}}<t\overset{\eqref{eqn: Galois connection less equal}}\leq G_r(x)})
        =G_r(x)-G_r(z).
    \end{align*}
    Similarly, $G_{\pushforward{\Tilde{H}}{\Lebesgue_H}}(x)=G_r(x)-G_r(z)$ for $z\geq x\in I_G$. Thus, we have found $c(z)=-G_r(z)$ as $G_{\pushforward{\Tilde{H}}{\Lebesgue_H}}=G_r-G_r(z)$ and therefore $\pushforward{\Tilde{H}}{{\Lebesgue}_{H}}=\mu_G$. As $G$ is a generalized inverse of $H$ also $\pushforward{\Tilde{G}}{\Lebesgue_G}=\mu_H$ follows.
\end{proof}

While \eqref{eqn: Qunatile generates measure} always holds, this is not the case in the reverse direction: If we try to push forward $\mu_H$ to $\Lebesgue_G$ by $H$ we need further regularity of $H$, namely that $\Tilde{H}$ is continuous.
\begin{lemma}
    \label{lem: continuity of generalized inverse}
    For a generalized inverse pair $G,H$ the following are
    equivalent
    \begin{enumerate}[left=\parindent]
        \item\label{enum: Tilde H continuous}
              $\Tilde{H}:I_H\to\R$ is continuous.
        \item \label{enum: H surjective}
              $H$ is surjective onto $M_G$.
        \item \label{enum: G injective}
              $G$ is injective on $M_G$.
        \item \label{enum: G strictly increasing}
              $G$ is strictly increasing on $M_G$.
    \end{enumerate}
    In each case $\restr{H\circ G}{M_G}=\id_{M_G}$ and $M_G=\interior{\Functionimage{H}{I_H}}$.
\end{lemma}
\begin{proof}\leavevmode
    \begin{itemize}[wide, labelwidth=!, labelindent=0pt,itemindent=!]
        \item[\enquote{\ref{enum: Tilde H continuous}$\Rightarrow$\ref{enum: H surjective},\ref{enum: G injective}}:]
              It holds that $\Tilde{H}$ being continuous implies $\restr{H\circ G}{\preimage{G}{I_H}}=\id_{\preimage{G}{I_H}}$ by \cite[Lemma 1.(l)]{wackerPleaseNotAnother2023} and thus $G$ is injective on $M_G=\interior{\preimage{G}{I_H}}$ and $H$ is surjective onto $M_G$, as there exists a left- respectively right inverse. Further as $\Tilde{H}$ is continuous it maps intervals to intervals and thus
              \begin{equation*}
                  M_G=\big(\inf_{t\in I_H}H(t),\sup_{t\in I_H}H(t)\big)=\interior{\Functionimage{\Tilde{H}}{I_H}}=\interior{\Functionimage{H}{I_H}}.
              \end{equation*}
        \item[\enquote{\ref{enum: G injective}$\Rightarrow$\ref{enum: G strictly increasing}}:]
              If $G$ is injective on $M_G$, then it is nowhere constant so strictly increasing.
        \item[\enquote{\ref{enum: G strictly increasing}$\Rightarrow$\ref{enum: Tilde H continuous}}:]
              We claim that, if $G$ is strictly increasing on $M_G$, then for all $t\in I_H$ holds $\interior{\preimage{G}{\set{t}}}=\emptyset$. Otherwise, there would be $x<y\in\interior{\preimage{G}{\set{t}}}$ with $t=G(x)=G(y)$ so $\set{x,y}\not\subset M_G$ but $\set{x,y}\subset\interior{\preimage{G}{\set{t}}}\subset\interior{\preimage{G}{I_H}}=M_G$.
              Therefore, we have for all $t\in I_H$ that $\setcount{\preimage{G}{\set{t}}}\leq 1$ which gives by \cite[(9)]{klementQuasiPseudoinversesMonotone1999} $H_l(t)=H_r(t)$ so $\Tilde{H}$ is continuous.\qedhere
    \end{itemize}
\end{proof}
\begin{lemma}
    \label{cor: pushforward with continuous cdf}
    For a generalized inverse pair $G,H$ holds if $\Tilde{H}$ is continuous then
    \begin{equation}
        \label{eqn: continuous cdf generates uniform}
        \pushforward{\Tilde{H}}{\mu_H}={\Lebesgue}_G.
    \end{equation}
\end{lemma}
\begin{proof}
    We use that $\Lebesgue_G(A)=\Lebesgue_G(A\cap M_G)$ for all $A\in\B(I_G)$ and thus
    \begin{equation*}
        \pushforward{\Tilde{H}}{\mu_H}(A)
        \overset{\eqref{eqn: Qunatile generates measure}}{=}\pushforward{(\Tilde{H}\circ\Tilde{G})}{{\Lebesgue}_{G}}(A)
        \overset{\autoref{lem: continuity of generalized inverse}}{=} \pushforward{(\id_{M_G})}{\Lebesgue_G}(A\cap M_G)
        ={\Lebesgue}_G(A).\qedhere
    \end{equation*}
\end{proof}
\begin{remark}
    \autoref{lem:pushforward of uniform with generalized inverse} generalizes that the \ac{QF} $Q_\mu:(0,1)\to\R$ is a random variable with distribution $\mu$ from the probability space $((0,1),\B(0,1),\uniformmeasure)$ to $(\R,\B(\R))$ \cite[Proposition 2.2.]{Santambrogio.2020}. Similarly, \autoref{cor: pushforward with continuous cdf} is known in the distribution case. If the \ac{CDF} $F_\mu:\R\to[0,1]$ of the distribution $\mu:\B(\R)\to[0,1]$ is continuous then it pushes $\mu$ forward to the uniform distribution on $(0,1)=M_{Q_\mu}$ \cite[Chapter 1.5]{Panaretos.2020} as $I_{F_\mu}=\R$ and consequently $I_{Q_\mu}=M_{Q_\mu}$.
\end{remark}
Applying the definition scheme of $I_G,M_G$ again we end up with the \emph{supporting interval} $S_G$ defined slightly altered as
\begin{equation*}
    S_G=\big[\inf_{y\in M_H}H(y),\sup_{y\in M_H}H(y)\big].
\end{equation*}
This time $S_G=\closure{\preimage{G}{M_H}}$ and consequently for all $x\in I_G\setminus S_G$ it holds either $G(x)=\inf M_H=\inf_{y\in I_G}G(y)$ or $G(x)=\sup M_H=\sup_{y\in I_G}G(y)$ (\autoref{fig: generalized inverse}). Therefore, $\Tilde{G}$ is constant outside $S_G$ and consequently $\mu_G(I_G\setminus S_G)=0$ and $S_G$ is the smallest closed interval with this property as for all $x\in\interior{S_G}$ it holds that there are $u<x<s$ such that $G(u)<G(x)<G(s)$. Therefore, $S_G$ is the convex hull of $\support{\mu_G}$, the support of $\mu_G$. If the support is convex, \ac{ie}, itself an interval, we have that $\support{\mu_G}=S_G$.
\section{Absolute continuity of the generalized inverse}
\label{sec: Abosultely continuity of the generalized inverse}
In this section we turn our focus to absolute continuity of measures and their distribution functions. After introducing both we find similar to \autoref{lem: continuity of generalized inverse} conditions when a generalized inverse is absolutely continuous. The Radon-Nikodym derivative fulfills an inverse function rule in that case.

From now on, non-decreasing functions $G:\R\to\Rbar$ are given as $\Tilde{G}:I_G\to\R$ and we will drop the tilde notation.
\subsection{Absolute continuity of measures}
For measures $\mu,\rho:\B(I)\to\Rbar_{\geq0}$ on some open interval $I\subseteq\R$, we call $\mu$ \emph{absolutely continuous \ac{wrt}} $\rho$, notated as $\mu\ll\rho$, if and only if for all $A\in\B(I)$ it holds that if $\rho(A)=0$ then also $\mu(A)=0$. Let $\mu$ and $\rho$ now be locally finite, then $\mu\ll\rho$ holds if and only if there exists a \emph{density} or \emph{Radon-Nikodym derivative} $\RadonNikodym{\mu}{\rho}:I\to\Rbar_{\geq0}$ of $\mu$ \ac{wrt} $\rho$ \cite[(19.24)]{hewittRealAbstractAnalysis1969}, \ac{ie}, \begin{equation*}
    \int_A \RadonNikodym{\mu}{\rho}\diff\rho=\mu(A),\quad\forall A\in\B(I).
\end{equation*}
All Radon-Nikodym derivatives of $\mu$ \ac{wrt} $\rho$ ensemble into an equivalence class of $\rho$-integrable functions, and integrals \ac{wrt} $\diff\mu$ are the same as integrals \ac{wrt} $\RadonNikodym{\mu}{\rho}\diff\rho$, \ac{ie}, a Radon-Nikodym derivative is really a density \cite[3.9 Proposition.a.]{follandRealAnalysisModern1999}.
Recall that $\mu$ has a unique \emph{Lebesgue decomposition} \ac{wrt} $\rho$ \cite[(19.42)]{hewittRealAbstractAnalysis1969} given by the measures $\mu_{abs}:\B(I)\to\Rbar_{\geq0}$ and $\mu_{\perp}:\B(I)\to\Rbar_{\geq0}$, where $\mu_{abs}\ll\rho$ and $\mu_{\perp}$ is \emph{singular \ac{wrt}} $\rho$, notated as $\mu_\perp\perp\rho$, \ac{ie}, there exists a \emph{singularity set} $\setletter_{\mu_{\perp{}}}\in\B(I)$ such that $\mu_{\perp}(I\setminus\setletter_{\mu_{\perp}})=0=\rho(\setletter_{\mu_\perp})$.
Given that singularity set, we also have the equalities $\mu_{abs}(A)=\mu(A\setminus \setletter_{\mu_{\perp{}}})$ and $\mu_{\perp}(A)=\mu(A\cap \setletter_{\mu_{\perp{}}})$ for all $A\in\B(I)$. We now characterize absolute continuity between $\mu$ and $\rho$ by properties of the Lebesgue decomposition of $\mu$ \ac{wrt} to $\rho$.
\begin{lemma}
    \label{Lem: Radon-Nikodym derivatives}
    Let $\mu,\rho:\B(I)\to\Rbar_{\geq0}$ be locally finite Borel measures with the Lebesgue decomposition $\mu=\mu_{abs}+\mu_\perp$ of $\mu$ \ac{wrt} $\rho$. Then
    \begin{enumerate}[left=\parindent]
        \item \label{enum: absolute continuity wrt to Lebesgue decomposition absolute continuity}
              $\mu\ll\rho$ if and only if $\mu_\perp=0$,
        \item \label{enum: absolute continuity wrt to Lebesgue decomposition reverse absolute continuity}
              $\rho\ll\mu$ if and only if
              \begin{enumerate}[i)]
                  \item \label{enum: asbsolute continuous wrt to absolute continuous part}
                        $\rho\ll\mu_{abs}$, or
                  \item \label{enum: absolute continuous density positive}
                        $\RadonNikodym{\mu_{abs}}{\rho}>0$ $\rho$-\ac{ae}.
              \end{enumerate}
    \end{enumerate}
\end{lemma}
\begin{proof}\leavevmode
    \begin{itemize}[wide, labelwidth=!, labelindent=0pt,itemindent=!]
        \item[\ref{enum: absolute continuity wrt to Lebesgue decomposition absolute continuity}]
              As the Lebesgue decomposition is unique, both cases are equivalent to $\mu=\mu_{abs}$.
        \item[\enquote{\ref{enum: absolute continuity wrt to Lebesgue decomposition reverse absolute continuity}$\LeftRightarrow$\ref{enum: asbsolute continuous wrt to absolute continuous part}}:] \leavevmode
              \begin{itemize}[wide = \parindent]
                  \item[\enquote{$\Rightarrow$}:]
                        Assume that $\rho\ll\mu$. Then, if $0=\mu_{abs}(A)=\mu(A\setminus \setletter_{\perp})$, we have by the absolute continuity that also $0=\rho(A\setminus \setletter_{\perp})=\rho(A)$ for all $A\in\B(I)$, with the last equality given by the definition of $\setletter_{\perp}$, ensuring that $\rho(\setletter_{\perp})=0$.
                  \item[\enquote{$\Leftarrow$}:]
                        Assume that $\rho\ll\mu_{abs}$. Now, if $\mu(A)=0$ for some $A\in\B(I)$ then so is $\mu_{abs}(A)$ and therefore $\rho(A)$.
              \end{itemize}
        \item[\enquote{\ref{enum: absolute continuous density positive}$\LeftRightarrow$\ref{enum: asbsolute continuous wrt to absolute continuous part}}:] If we assume that both measures are equivalent, \ac{ie}, that $\mu_{abs}\ll\rho$ and $\rho\ll\mu_{abs}$, then $\mu_{abs}(\set{x\in I\given \RadonNikodym{\mu_{abs}}{\rho}=0})=0$ and thus $\rho(\set{x\in I\given \RadonNikodym{\mu_{abs}}{\rho}=0})=0$, \ac{ie}, $\RadonNikodym{\mu_{abs}}{\rho}>0$ $\rho$-\ac{ae}.
              Therefore, in both cases we have that $\frac{1}{\RadonNikodym{\mu_{abs}}{\rho}}\RadonNikodym{\mu_{abs}}{\rho}=1$ $\rho$-\ac{ae}, as any Radon-Nikodym derivative can be chosen finite.
              By calculating that $\frac{1}{\RadonNikodym{\mu_{abs}}{\rho}}$ is a Radon-Nikodym derivative of $\rho$ \ac{wrt} $\mu_{abs}$ we prove that $\rho\ll\mu_{abs}$
              \begin{equation*}
                  \int_A\frac{1}{\RadonNikodym{\mu_{abs}}{\rho}}\diff\mu_{abs}=\int_A\frac{1}{\RadonNikodym{\mu_{abs}}{\rho}}\RadonNikodym{\mu_{abs}}{\rho}\diff\rho=\rho(A),\quad\forall A\in\B(I).\qedhere
              \end{equation*}
    \end{itemize}
\end{proof}
\subsection{Absolute continuity of non-decreasing real-valued functions}\leavevmode{}
\begin{definition}\label{def:Radon-Nikodym derivative}
    A non-decreasing real-valued function $G:I_G\to\R$ is \emph{absolutely continuous} if and only if $\mu_G\ll\restr\Lebesgue{I_G}$. A function $g:I_G\to\Rbar_{\geq0}$ is the \emph{Radon-Nikodym derivative} of $G$ if $g=\RadonNikodym{\mu_G}{\restr{\Lebesgue}{I_G}}$ $\restr{\Lebesgue}{I_G}$-\ac{ae}. For an open subinterval $I\subset I_G$ we have $G$ is \emph{absolutely continuous on} $I$ if and only if $\restr{G}{I}$ is absolutely continuous.
\end{definition}
\begin{remark}
    This definition is equivalent to the non-decreasing real-valued function being absolute continuous in the analysis sense \cite[7.18]{Rudin.1987b} recalled in the proof of \autoref{thm: absolute continuity of the generalized inverse}. In particular, absolutely continuity implies continuity. The definition is also equivalent to the definition mentioned in the introduction, which in contrast does not reference the associated measure $\mu_G$.
\end{remark}
There is also a Lebesgue decomposition of a non-decreasing function \cite[(19.61)]{hewittRealAbstractAnalysis1969}
\begin{equation*}
    G_r=G_{abs}+G_{\perp}+G_r(z)\defeq G_{(\mu_G)_{abs}}+G_{(\mu_G)_\perp}+G_r(z)
\end{equation*}
corresponding to the Lebesgue decomposition $\mu_G$ \ac{wrt} to $\restr{\Lebesgue}{I_G}$ via the distribution functions defined by \eqref{eq: monotone function from measure} and a fixed choice of $z\in I_G$. It separates $G_r$ into an absolute continuous part $G_{abs}$ and a \emph{singular} part $G_{\perp}$, \ac{ie}, $G_\perp$ is $\Lebesgue$-\ac{ae} differentiable with derivative $0$, and a constant shift based on $z$. Conversely, $\mu_G=\mu_{G_{abs}}+\mu_{G_\perp}$ is always a Lebesgue decomposition of $\mu_G$ \ac{wrt} $\restr{\Lebesgue}{I_G}$.
Actually, the former is also the Lebesgue decomposition of $\mu_G$ \ac{wrt} $\Lebesgue_G$ the Lebesgue measure supported on the mass interval $M_G$. We directly have that $\mu_{G_{abs}}\ll\Lebesgue_G$ as $\mu_{G_{abs}}(I_G\setminus M_G)=0$ and $\mu_{G_\perp}\perp\Lebesgue_G$ as $\Lebesgue_G\ll\restr{\Lebesgue}{I_G}$.\footnote{These calculations also hold for degenerate $M_G=\emptyset$ and $\Lebesgue_G=0:\B(I_G)\to\set{0}$ the zero measure as in that case $\mu_{G_{abs}}=0=\Lebesgue_G$.} Further, by the chain rule for Radon-Nikodym derivatives \cite[3.9 Proposition.b.]{follandRealAnalysisModern1999} we have that
\begin{equation*}
    g_{abs}\defeq\RadonNikodym{\mu_{G_{abs}}}{\restr{\Lebesgue}{I_G}}
    =\RadonNikodym{\mu_{G_{abs}}}{{\Lebesgue}_{G}}\RadonNikodym{{\Lebesgue}_{G}}{\restr{\Lebesgue}{I_G}}
    =\RadonNikodym{\mu_{G_{abs}}}{{\Lebesgue}_{G}}\cdot\ind_{M_G}\quad\text{$\restr{\Lebesgue}{I_G}$-\ac{ae}}
\end{equation*}
where $\ind_{M_G}:I_G\to\set{0,1}$ is the indicator function of $M_G$ and as $\ind_{M_G}=1$ $\Lebesgue_G$-\ac{ae}
\begin{equation}
    \label{eqn: relation of Radon-Nikodym derivatives}
    g_{abs}=\RadonNikodym{\mu_{G_{abs}}}{{\Lebesgue}_{G}}\quad\text{${\Lebesgue}_{G}$-\ac{ae}.}
\end{equation}
\subsection{Absolute continuity of the generalized inverse}
Just like the differentiability of a regular inverse the generalized inverse $H$ is not directly absolutely continuous only because $G$ is.
Instead, the absolute continuity of a generalized inverse $H:I_H\to\R$ depends on the Radon-Nikodym derivative $g_{abs}$ on the mass interval $M_G$.

\begin{theorem}[adapted from {\cite[Proposition A.17]{bobkovOnedimensionalEmpiricalMeasures2019}}]
    \label{thm: absolute continuity of the generalized inverse}
    For a generalized inverse pair $G,H$, an open subinterval $I\subseteq I_H$, the corresponding $M\defeq\interior{\preimage{G}{I}}$ and the restricted Lebesgue measure $\restr{\Lebesgue}{M}:\B(I_G)\to\Rbar_{\geq0}$ with support $\closure{M}$ the following are equivalent:
    \begin{enumerate}[left=\parindent]
        \item $H$ is absolutely continuous on $I$.
              \label{enum: absolute continuous generalized inverse}
        \item A Radon-Nikodym derivative $g_{abs}$ of $G_{abs}$ satisfies $g_{abs}>0$ $\restr{\Lebesgue}{M}$-\ac{ae}.
              \label{enum: positive density}
        \item  It holds $\restr{\Lebesgue}{M}\ll\mu_G$.
              \label{enum: Lebesgue is absolutely continuous}
    \end{enumerate}
\end{theorem}
\begin{proof}
    Assume \ac{wlog} that $I\neq\emptyset$.
    In the case where $\Functionimage{H}{I}$ is a singleton and $M=\emptyset$ all three properties are always true.
    \begin{itemize}[wide, labelwidth=!, labelindent=0pt,itemindent=!]
        \item[{\ref{enum: absolute continuous generalized inverse}}]
              The generalized inverse $H$ is constant on $I$, so absolutely continuous.
        \item[\ref{enum: positive density}]
              As $\restr{\Lebesgue}{M}$  is the zero measure $0:\B(I_G)\to\set{0}$, all properties of $g_{abs}$ hold $\restr{\Lebesgue}{M}$-\ac{ae}.
        \item[\ref{enum: Lebesgue is absolutely continuous}]
              By the same reason, $\restr{\Lebesgue}{M}\ll\mu$ for all Borel measures $\mu:\B(I_G)\to\Rbar_{\geq0}$.

    \end{itemize}
    The non-degenerate case is treated in \cite{bobkovOnedimensionalEmpiricalMeasures2019} where the equivalence of \ref{enum: absolute continuous generalized inverse} and \ref{enum: positive density} is stated as Proposition A.17 in the case where $G$ is a \ac{CDF} of a non-degenerate probability measure and $I=I_H$. We give a slightly adapted proof here for completeness’s sake:

    Assume \ac{wlog} that $I=I_H$, since we can otherwise consider $\widehat{G}$ a generalized inverse of $\restr{H}{I}:I\to\R$ together with the latter.
    Then, we have that $\widehat{G}_r=G_r$ on $M$ and therefore $g_{abs}=\widehat{g}_{abs}$ $\restr{\Lebesgue}{M}$-\ac{ae}.
    Also, $\restr{\mu_G}{M}=\restr{\mu_{\widehat{G}}}{M}$ and consequently $\restr{\Lebesgue}{M}\ll\restr{\mu_{\widehat{G}}}{M}\ll\mu_{\widehat{G}}$ if and only if $\restr{\Lebesgue}{M}\ll\restr{\mu_{G}}{M}\ll\mu_{G}$.
    Thus, we take for granted that $I=I_H$, $H$ is not constant, $M=M_G$ is the mass interval of $G$ and $\restr{\Lebesgue}{M}=\Lebesgue_G$.
    \begin{itemize}[wide, labelwidth=!, labelindent=0pt,itemindent=!]
        \item[\enquote{\ref{enum: Lebesgue is absolutely continuous}$\LeftRightarrow$\ref{enum: positive density}}:]
              Since $\mu_G=\mu_{G_{abs}}+\mu_{G_{\perp}}$ is also the Lebesgue decomposition of $\mu_G$ \ac{wrt} $\Lebesgue_G$, we have by \autoref{Lem: Radon-Nikodym derivatives} that $\Lebesgue_G\ll\mu_G$ if and only if $0<\RadonNikodym{\mu_{G_{abs}}}{\Lebesgue_G}
                  \overset{\eqref{eqn: relation of Radon-Nikodym derivatives}}
                  {=}g_{abs}$ $\Lebesgue_G$-\ac{ae}.
        \item[\enquote{\ref{enum: absolute continuous generalized inverse}$\LeftRightarrow$\ref{enum: Lebesgue is absolutely continuous}}:]
              In both cases $H:I_H\to\R$ is continuous. Under \ref{enum: absolute continuous generalized inverse} this is clear and under \ref{enum: Lebesgue is absolutely continuous} it holds by \ref{enum: positive density} that $G_{abs}$ is strictly increasing on $M_G$ and consequently by \autoref{lem: continuity of generalized inverse} $H$ is continuous. By the continuity of $H$ we further have that $M_G=\interior{\Functionimage{H}{I_H}}$.

              The function $H:I_H\to\R$ is \emph{absolutely continuous in the analysis sense} \cite[7.17]{Rudin.1987b} on $I_H$ if and only if for every $[a,b]\subset I_H$ and $\epsilon>0$ there exists a $\delta>0$ such that for any disjoint sequence of open subintervals $((a_j,b_j))_{j\in N}$ of $[a,b]$ with possible countable index set $N\subseteq\N$ it holds
              \begin{equation*}
                  \sum_{j\in N} b_j-a_j<\delta\Rightarrow\sum_{j\in N} H(b_j)-H(a_j)<\epsilon.
              \end{equation*}
              As $H$ is not constant, we may assume that $H(b_j)\neq H(a_j)$ for all $j\in N$.
              We can replace $a_j$ and $b_j$ by their largest respectively their smallest value such that their value under $H$ does not change. Set $x_j=H(a_j)$ and $y_j=H(b_j)$ for $j\in N$, then the replacements are $G_r(x_j)\geq a_j$ and $G_l(y_j)\leq b_j$. By the continuity of $H$ the closed subsets $[a,b]\subset I_H$ generate all closed subsets $[x,y]=\Functionimage{H}{[a,b]}\subseteq\Functionimage{H}{I_H}$. Therefore, we have that the absolute continuity of $H$ is equivalent to the fact that for every $[x,y]\subseteq\Functionimage{H}{I_H}$ and $\epsilon>0$ there exists a $\delta>0$ such that for every sequence of disjoint subintervals $((x_j,y_j))_{j\in N}$ of $[x,y], N\subseteq\N$ it holds
              \begin{equation*}
                  \delta>\sum_{j\in N} G_l(y_j)-G_r(x_j)=\mu_G\bigg(\bigcup_{j\in N}(x_j,y_j)\bigg)\Rightarrow\epsilon>\sum_{j\in N} y_j-x_j=\Lebesgue\bigg(\bigcup_{j\in N}(x_j,y_j)\bigg).
              \end{equation*}
              This in turn is equivalent to the property that for every $[x,y]\subseteq\Functionimage{H}{I_H}$ and $\epsilon>0$ there exists a $\delta>0$ such that $\mu_G(A)<\delta\Rightarrow\Lebesgue(A)<\epsilon $ for all $A\in\B([x,y])$, by regularity of measures and the fact that $\Lebesgue(\set{x,y})=0$.
              As $\restr{\Lebesgue}{[x,y]}$ is a finite measure, this is the same \cite[3.5 Theorem]{follandRealAnalysisModern1999} as $\restr{\Lebesgue}{[x,y]}\ll\restr{\mu_G}{[x,y]}$ for all subintervals $[x,y]\subseteq \Functionimage{H}{I_H}$ so the same as $\restr{\Lebesgue}{\Functionimage{H}{I_H}}\ll\restr{\mu_G}{\Functionimage{H}{I_H}}$.
              Finally, as $\Functionimage{H}{I_H}\supset{M_G}$ and thus $\Lebesgue_G(I_G\setminus \Functionimage{H}{I_H})\leq\Lebesgue_G(I_G\setminus M_G)=0$, we have that $\restr{\Lebesgue}{\Functionimage{H}{I_H}}\ll\restr{\mu_G}{\Functionimage{H}{I_H}}$ if and only if $\Lebesgue_{G}\ll\mu_G$.\qedhere
    \end{itemize}
\end{proof}

We now set out to show the aforementioned inverse function rule for the Radon-Nikodym derivative.

\begin{restatable}{theorem}{Lem: Radon-Nikodym derivative of generalized inverse}
    \label{thm: Radon-Nikodym derivative generlized inverse function}
    In the setting of \autoref{thm: absolute continuity of the generalized inverse} we have that if $H:I_H\to\R$ is absolutely continuous on an open subinterval $I\subseteq I_H$ with a Radon-Nikodym derivative $h:I\to\Rbar_{\geq0}$, then
    \begin{equation}
        g_{abs}
        =\frac{1}{h\circ G}\quad\text{$\restr{\Lebesgue}{M}$-\ac{ae}.}\label{eqn: density of absolutely continuous part of inverse measure of absolutely continuous measure}
    \end{equation}
\end{restatable}
\begin{proof}
    Again in the case where $\Functionimage{H}{I}$ is a singleton this holds trivially as $\restr{\Lebesgue}{M}$ is the zero measure.
    Thus, we assume that $\Functionimage{H}{I}$ is a proper interval and again \ac{wlog} that $I=I_H$ and $\restr{\Lebesgue}{M}=\Lebesgue_G$.
    Instead of proving \eqref{eqn: relation of Radon-Nikodym derivatives} we will prove the reciprocal equation in two steps
    \begin{equation*}
        \frac1{g_{abs}}\overset{1}{=}\RadonNikodym{{\Lebesgue}_G}{\mu_G}\overset{2}{=}h\circ G\quad\text{${\Lebesgue}_G$-\ac{ae}.}
    \end{equation*}
    \begin{itemize}[wide, labelwidth=!, labelindent=0pt,itemindent=!]
        \item[\enquote{$\overset{1}{=}$}:]
              As $0<g_{abs}=\RadonNikodym{\mu_{G_{abs}}}{{\Lebesgue}_{G}}$ $\Lebesgue_G$-\ac{ae} by \autoref{thm: absolute continuity of the generalized inverse}, we conclude from the proof of \autoref{Lem: Radon-Nikodym derivatives} that \begin{equation*}
                  \RadonNikodym{\Lebesgue_G}{\mu_{G_{abs}}}=\frac{1}{\RadonNikodym{\mu_{G_{abs}}}{\Lebesgue_G}}=\frac1{g_{abs}}\quad\text{$\Lebesgue_G$-\ac{ae}}
              \end{equation*}
              By definition of $\mu_{G_{abs}}=(\mu_G)_{abs}$ we have that for all $A\in\B(I_G)$
              \begin{equation*}
                  \mu_{G_{abs}}(A)=\mu_G(A\setminus\setletter_{{\perp}})=\int_A\ind_{I_G\setminus\setletter_{{\perp}}}\diff\mu_G
              \end{equation*}
              such that $\RadonNikodym{\mu_{G_{abs}}}{\mu_G}=\ind_{I_G\setminus\setletter_{\perp}}$ $\mu_G$-\ac{ae}.
              Now, by the chain rule for Radon-Nikodym derivatives \cite[3.9 Proposition.b.]{follandRealAnalysisModern1999} it holds that
              \begin{equation*}
                  \RadonNikodym{\Lebesgue_G}{\mu_{G}}=\RadonNikodym{\Lebesgue_G}{\mu_{G_{abs}}}\RadonNikodym{\mu_{G_{abs}}}{\mu_G}=\RadonNikodym{\Lebesgue_G}{\mu_{G_{abs}}}\cdot\ind_{I_G\setminus\setletter_{\perp}}\quad\text{$\mu_G$-\ac{ae}.}
              \end{equation*}
              As $\mu_{G_{abs}}(A_\perp)=0$ we have that $0=\Lebesgue_{G}(A_\perp)=\Lebesgue(A_\perp\cap M_G)$ and therefore holds $\ind_{I_G\setminus\setletter_{\perp}}=\ind_{M_G}=1$ $\Lebesgue_G$-\ac{ae}. Consequently, as $\Lebesgue_G\ll\mu_G$, we have that
              \begin{equation*}
                  \RadonNikodym{\Lebesgue_G}{\mu_{G}}=\RadonNikodym{\Lebesgue_G}{\mu_{G_{abs}}}
                  =\frac{1}{g_{abs}}\quad\text{$\Lebesgue_G$-\ac{ae}.}
              \end{equation*}
        \item[\enquote{$\overset{2}{=}$}:]To continue, we need the equality
              \begin{equation}
                  h=\ind_{\Functionimage{H}{I_H\setminus C_H}}\circ H\cdot(h\circ G\circ H)\quad\text{${\Lebesgue}_{H}$-\ac{ae}},
                  \label{eqn: necessary comperativitiy of h and G circ H}
              \end{equation}
              where $C_H\subseteq I_H$ denotes the set of points where $H$ is constant, \ac{ie}, not injective.
              Naturally, $H$ is injective on $I_H\setminus C_H$ and $G$ consequently its left inverse there, so that \eqref{eqn: necessary comperativitiy of h and G circ H} holds on $I_H\setminus C_H$ even everywhere.
              Now, as $H$ is constant on $C_H$, we have that $h=0$ $\Lebesgue$-\ac{ae} on $C_H$ and so is the right-hand side of \eqref{eqn: necessary comperativitiy of h and G circ H} as $\Functionimage{H}{I_H\setminus C_H}\cap\Functionimage{H}{C_H}=\emptyset$.
              Using this, we now show a second characterization of $\RadonNikodym{{\Lebesgue}_G}{\mu_G}$.
              To this end, take an arbitrary $A\in\B(I_G)$, then it holds
              \begin{align*}
                  {\Lebesgue}_G(A)
                   & =\pushforward{H}{\mu_H}(A)
                  =\int_{\preimage{H}{A}}\diff\mu_H
                  =\int_{\preimage{H}{A}}h\diff{\Lebesgue}_H                                                                             \\
                  \overset{\eqref{eqn: necessary comperativitiy of h and G circ H}}
                   & =\int_{\preimage{H}{A}}\big(\ind_{\Functionimage{H}{I_H\setminus C_H}}\cdot(h\circ G)\big)\circ H\diff{\Lebesgue}_H \\
                   & =\int_A  \ind_{\Functionimage{H}{I_H\setminus C_H}}\cdot (h\circ G)\diff\pushforward{H}{{\Lebesgue}_H}
                  =\int_A \ind_{\Functionimage{H}{I_H\setminus C_H}}\cdot (h\circ G)\diff\mu_G
              \end{align*}
              and thus $\RadonNikodym{{\Lebesgue}_G}{\mu_G}=\ind_{\Functionimage{H}{I_H\setminus C_H}}\cdot (h\circ G)$ $\mu_G$-\ac{ae}. As the constant regions of $H$ correspond to the jumps of $G$ \cite[(3.15),(3.16)]{delafortelleStudyGeneralizedInverses2015}, $\Functionimage{H}{C_H}$ is the at most countable set of discontinuities of $G$ and has Lebesgue measure $0$. Consequently, we have \begin{equation*}
                  \Lebesgue_G(I_G\setminus\Functionimage{H}{I_H\setminus C_H})=\Lebesgue_G(I_G\setminus M_G\cup \Functionimage{H}{C_H})=0
              \end{equation*} and arrive at
              \begin{equation*}
                  h\circ G=\RadonNikodym{{\Lebesgue}_G}{\mu_G}\quad\text{${\Lebesgue}_G$-\ac{ae}.}
              \end{equation*}
              Combining the results and inverting the equation yields the result.\qedhere
    \end{itemize}
\end{proof}
\begin{remark}
    In \cite[Proposition A.19]{bobkovOnedimensionalEmpiricalMeasures2019} the reverse equation is shown for a specific representative of $g_{abs}$, \ac{ie},
    \begin{equation*}
        h(t)=\frac{1}{g_{abs}(H(t))},\quad \forall t\in I_H\setminus C_H,
    \end{equation*}
    defines a Radon-Nikodym derivative of $H$. The proof is based on working with generalized derivatives as all absolutely continuous functions are $\Lebesgue$-\ac{ae} differentiable on their domain and the Radon-Nikodym derivatives agree again $\Lebesgue$-\ac{ae} with this derivative.
\end{remark}
\section{Unimodality of locally finite measures}
\label{sec: Unimodality of locally finite measures}
Before we prove \autoref{thm:unimodaldistributionhasabsolutelycontinuousqunatilefunction} and \autoref{thm: Unimodal distributions via quantile} we first extend the definitions of unimodality to locally finite measures on $\R$. Here, we restrict ourselves to non-decreasing real-valued functions on $\R$, so $G:\R\to\R$ with $I_G=\R$ which emulates the setting of \acp{CDF}{}.
\begin{definition}[Compare \autoref{def:density unimodal} and \autoref{def: Unimodality CDF}]
    \label{def: unimodality-generating}
    A non-decreasing function $G:\R\to\R$ is \emph{unimodality-generating} if there is a $\nu\in\Rbar$ such that $\restr{G}{(-\infty,\nu]}$ is convex and $\restr{G}{[\nu,+\infty)}$ is concave. We call any such $\nu$ a \emph{mode} of $G$.
    A locally finite Borel measure $\mu:\B(\R)\to\Rbar_{\geq0}$ is \emph{\ac{CDF}-unimodal} if any or equivalently all its distribution functions $G_\mu:\R\to\R$ are unimodality-generating and \emph{dens-unimodal} if $\mu$ is absolutely continuous with at least one quasi-concave density $g_\mu$. Any modes of $G_\mu$ respectively of $g_\mu$ are again called modes of $\mu$.
\end{definition}
\begin{remark}
    We choose the term unimodality-generating to emphasize the fact the associated measure $\mu_G$ it generates is unimodal. Still, every unimodality-generating function is indeed quasi-convex and quasi-concave as it is non-decreasing and therefore itself unimodal.
\end{remark}

\subsection{Relation c.d.f.-unimodality and dens-unimodality}
\label{sec: relation CDF unimodality and dens Unimodality}
Convexity and concavity of a general function $V:D\to\R$ defined on an interval $D$ satisfy the well known inequality $V(\convexindex x + (1-\convexindex) y)\leq$, respectively $\geq$, $\convexindex V(x)+(1-\convexindex)V(y)$ for all $\convexindex\in(0,1),x,y\in D$. On open intervals these functions are very regular.
\begin{theorem}[{\cite[12 Theorem A]{Roberts.1973}}]
    \label{thm: convexity characterisation}
    A function $V:I\to\R$ defined on an open interval $I\subseteq\R$ is convex (concave) if and only if $V$ is absolutely continuous and there is a Radon-Nikodym derivative $v$ of $V$ which is non-decreasing (non-increasing).
\end{theorem}
Further, if $V:(a,b)\to\R$, $a<b\in\R$ is convex then $V:[a,b]\to\Rbar$ is convex if and only if $V(a)\geq V(a+),V(b)\geq V(b-)$, respectively with $\leq$ for concavity.
Equipped like this we now give the relation between \ac{CDF}-unimodal and dens-unimodal measures.
\begin{lemma}
    \label{lem: Relation dens unimodal and cdf unimodal}
    A locally finite measure $\mu:\B(\R)\to\Rbar_{\geq0}$ is \ac{CDF}-unimodal if and only if $\mu_{abs}$ is dens-unimodal and $\mu_\perp=c\delta_{\nu},c\in\R_{\geq0}$ is a rescaled Dirac measure at a mode $\nu$ of $\mu_{abs}$ (or the $0$ measure for $c=0$), where $\mu=\mu_{abs}+\mu_\perp$ is the Lebesgue decomposition of $\mu$ \ac{wrt} to the Lebesgue measure $\Lebesgue:\B(\R)\to\Rbar_{\geq0}$.
\end{lemma}
\begin{proof} Let $G:\R\to\R$ be a right-continuous distribution function of $\mu$.
    \begin{itemize}[wide = \parindent]
        \item[\enquote{$\Rightarrow$}:]
              Assume that $G$ is unimodality-generating.
              Consequently, by \autoref{thm: convexity characterisation} $G$ is absolutely continuous on $\R\setminus\set{\nu}$ for some mode $\nu$ of $G$.
              If $\nu\not\in\R$, then $G$ is absolutely continuous and again by \autoref{thm: convexity characterisation} there is a Radon-Nikodym derivative $g$ which is non-decreasing or non-increasing and therefore quasi-concave and $\mu_\perp=0$ (\autoref{Lem: Radon-Nikodym derivatives}).
              Thus, we assume $\nu\in\R$ and take the Lebesgue decomposition of $G$ based on $z=\nu$ given as $G=G_{abs}+G_\perp+G(\nu)$. Because $G$ is unimodality-generating, we have $G=G_{abs}$ on $(-\infty,\nu)$, $G=G_{abs}+c$ on $(\nu,+\infty)$ and $\mu_{G_\perp}=c\delta_{\nu}$ for $c\defeq\mu_G(\set{\nu})$. So, $G_{abs}$ is convex on $(-\infty,\nu)$, respectively concave on $(\nu,+\infty)$ and there exists a Radon-Nikodym derivative $g_{abs}$ which is non-decreasing on $(-\infty,\nu)$ and non-increasing on $(\nu,+\infty)$ by \autoref{thm: convexity characterisation}. Finally, setting $g_{abs}(\nu)\defeq\sup_{x\in \R\setminus\set{\nu}}g_{abs}(x)$ we have that $g_{abs}$ is quasi-concave.
        \item[\enquote{$\Leftarrow$}:]
              For the reverse direction we again have that $G=G_{abs}$ on $(-\infty,\nu)$ and $G=G_{abs}+c$ on $(\nu,+\infty)$ where $c\defeq\mu_\perp(\set{\nu})$. So, $c=0$ if $\nu\not\in\R$.
              By the dens-unimodality of $\mu_{abs}$ there is a quasi-concave Radon-Nikodym derivative $g_{abs}$. From \autoref{thm: convexity characterisation} we conclude that $G_{abs}$ and consequently $G$ is convex on $(-\infty,\nu)$ and concave on $(\nu,+\infty)$.
              As $G$ is non-decreasing, the convexity extends to $(-\infty,\nu]$ and the concavity to $[\nu,+\infty)$, \ac{ie}, $G$ is unimodality-generating.\qedhere
    \end{itemize}
\end{proof}
\subsection{Absolute continuity of the q.f. of a c.d.f.-unimodal distribution}
\label{sec: Proof of absolute continuity o fgeneralized inverse of unimodal genearting}
We now return to probability distributions  $\mu:\B(\R)\to[0,1]$.
Recall, that the \ac{CDF} $F_\mu:\R\to\R$ of $\mu$ is the right-continuous version of the specific distribution function of $\mu$ which vanishes at $-\infty$, \ac{ie}, we have that $F_\mu(-\infty)=0$, and that $Q_\mu:(0,1)\to\R$ is the real restriction of the left-continuous version of the generalized inverse of $F_\mu$. Further, as we still have $I_{F_\mu}=\R$, we have $I_{Q_\mu}=M_{Q_\mu}=(0,1)$ and $\closure{M_{F_\mu}}=\closure{\preimage{F_\mu}{M_{Q_\mu}}}=S_{F_\mu}$ so $\Lebesgue_{F_\mu}$ is the Lebesgue measure restricted to the supporting interval $S_{F_\mu}$ on $\B(\R)$.
\begin{proof}[Proof of \autoref{thm:unimodaldistributionhasabsolutelycontinuousqunatilefunction}]
    If $\mu_{abs}=0$ then $Q_\mu\equiv\nu$, so $Q_\mu$ is constant and absolutely continuous. Therefore, assume now that $\mu_{abs}\neq0$.
    By \autoref{lem: Relation dens unimodal and cdf unimodal} the absolutely continuous part $\mu_{abs}$ has a quasi-concave density $f_{abs}$ and hence $\support{\mu_{abs}}=\closure{\preimage{f_{abs}}{\R_{>0}}}\neq\emptyset$ is convex. Further, $\support{\mu_\perp}=\set{\nu}$ with $\nu$ a mode of $f_{abs}$ and consequently $\support{\mu_\perp}\subseteq\support{\mu_{abs}}=\support{\mu}$. Therefore,
    \begin{equation*}
        \closure{\preimage{f_{abs}}{\R_{>0}}}=\support{\mu}=S_{F_\mu}.
    \end{equation*}
    This gives us $f_{abs}>0$ $\Lebesgue_{F_\mu}$-\ac{ae} from which we conclude by \autoref{thm: absolute continuity of the generalized inverse} that $Q_\mu$ is absolutely continuous.
\end{proof}
\subsection{Characterizations of c.d.f.-unimodality via the q.f}
\label{sec: Proof of characterization of Unimodality via quantile functions}
\begin{proof}[Proof of \autoref{thm: Unimodal distributions via quantile}]\leavevmode
    \begin{enumerate}[wide, labelwidth=!, labelindent=0pt,itemindent=!]
        \item In both cases we know that $Q_\mu:(0,1)\to\R$ is absolutely continuous once by assumption and once by \autoref{thm:unimodaldistributionhasabsolutelycontinuousqunatilefunction}.
              \begin{itemize}[wide = \parindent]
                  \item[\enquote{$\Leftarrow$}:]
                        Assume first that $q_\mu$ is quasi-convex and define based on \autoref{thm: Radon-Nikodym derivative generlized inverse function} a Radon-Nikodym derivative $f_{abs}$ of the absolutely continuous part $F_{\mu_{abs}}$ of $F_\mu$ as $f_{abs}\defeq\frac{1}{q_\mu\circ F_\mu}$ on $M_{F_\mu}$. As $\closure{M_{F_\mu}}=S_{F_\mu}=S_{F_{\mu_{abs}}}$, we can extend $f_{abs}$ with $0$ on $\R\setminus M_{F_\mu}$ to obtain a Radon-Nikodym derivative on $\R$. As this $f_{abs}$ vanishes outside $M_{F_\mu}$, we only need to check \eqref{def: quasi-concaity} on $M_{F_\mu}$ to prove the quasi-concavity of $f_{abs}$.
                        To this end, take $x<y\in M_{F_\mu}$, so $F_{\mu}(x)<F_{\mu}(y)$, and $\convexindex\in(0,1)$ and set $\convexindex'=\frac{F_\mu(y)-F_\mu(\convexindex x+ (1-\convexindex)y)}{F_\mu(y)-F_\mu(x)}\in(0,1)$. Then, it holds
                        \begin{align*}
                             &                 & q_\mu (\convexindex' F_\mu (x)+ (1-\convexindex')F_{\mu}(y))  & \leq \max({q_{\mu}(F_\mu(x))},{q_{\mu}(F_{\mu}(y))})        \\
                             & \Rightarrow     & q_\mu\circ F_\mu(\convexindex x+ (1-\convexindex)y)           & \leq \max(q_\mu\circ F_\mu(x),q_\mu\circ F_\mu(y))
                            \\
                             & \LeftRightarrow & \frac1{q_\mu\circ F_\mu}(\convexindex x + (1-\convexindex) y) & \geq\frac1{\max({q_\mu\circ F_\mu(x),q_\mu\circ F_\mu(y)})} \\
                             & \LeftRightarrow & f_{abs}(\convexindex x + (1-\convexindex) y)                  & \geq\min({f_{abs}(x),f_{abs}(y)}).
                        \end{align*}
                        So, $f_{abs}$ is quasi-concave on $M_{F_{\mu}}$, thus on $\R$ and $\mu_{abs}$ is dens-unimodal.
                        Now, if $F_\mu$ is not absolutely continuous on $\R$, then by \autoref{thm: absolute continuity of the generalized inverse} there is a proper interval on which $q_\mu=q_{abs}=0$ $\Lebesgue$-\ac{ae}.
                        By the quasi-convexity of $q_\mu$ we have that $\preimage{q_\mu}{\set{0}}$ is an interval and its closure $\closure{\preimage{q_\mu}{\set{0}}}\eqdef[\alpha_{\min},\alpha_{\max}]$ gives the quantile modal interval.
                        In particular, $Q_\mu$ is only constant on this one interval and the domain of $F_\mu$ decomposes into an absolutely continuous part followed by a single jump at $\nu\in\set{\nu}= \Functionimage{Q_\mu}{[\alpha_{\min},\alpha_{\max}]}$ and again an absolutely continuous part by \autoref{thm: absolute continuity of the generalized inverse}. If we now set $q_{\mu}(\alpha_{\max})=0$, which does not change the quasi-concavity, then $f_{abs}(\nu)=+\infty$ and $\nu$ is the mode of $f_{abs}$ and coincides with the jump at $\nu$. Therefore, $\mu$ is \ac{CDF}-unimodal by \autoref{lem: Relation dens unimodal and cdf unimodal}.
                  \item[\enquote{$\Rightarrow$}:]
                        Assume now that $\mu$ is \ac{CDF}-unimodal.
                        In the case where $F_\mu$ is absolutely continuous the same calculation as above and setting $q_\mu=q_{abs}\defeq\frac{1}{f_\mu\circ Q_\mu}$ (\autoref{thm: Radon-Nikodym derivative generlized inverse function}) would yield that $q_\mu$ is quasi-convex.
                        Consequently, we assume now that $\mu_\perp\neq0$ is a Dirac measure at $\nu$ a mode of the quasi-concave $f_{abs}$ by \autoref{lem: Relation dens unimodal and cdf unimodal}.
                        By the absolute continuity of $F_\mu$ on $(-\infty,\nu)$, respectively $(\nu,+\infty)$, we can define a valid quantile density  on $(0,1)\setminus[\alpha_{\min},\alpha_{\max}]=\preimage{Q_\mu}{\R\setminus\set{\nu}}$ by $q_\mu\defeq\frac1{f_{abs}\circ Q_\mu}$. Inheriting from $f_{abs}$, this $q_\mu$ is non-increasing on $(0,\alpha_{\min})$ and non-decreasing on $(\alpha_{\max},1)$. On $[\alpha_{\min},\alpha_{\max}]$ we have that $Q_\mu$ is constantly $\nu$, so we can set $q_\mu=0$ on $[\alpha_{\min},\alpha_{\max}]$ to obtain a quasi-convex function which is a Radon-Nikodym derivative of $Q_\mu$.
              \end{itemize}
        \item[\enquote{\ref{enum: quantile density definition of unimodality}$\LeftRightarrow$\ref{enum: Qunatife function definition of unimodality}}:] By \autoref{thm: convexity characterisation} the conditions on $Q_\mu$ directly give that it is absolutely continuous everywhere except at $\alpha$. By the convexity of $Q_\mu$ on $[\alpha,1)$, a discontinuous upward jump would be possible there, but $Q_\mu$ is non-decreasing. Thus, $Q_\mu$ is continuous at $\alpha$ and therefore everywhere absolutely continuous. So, in both cases $Q_\mu$ is everywhere absolutely continuous. Consequently, \autoref{thm: convexity characterisation} gives the equivalence between $Q_\mu$ being first concave, then convex and $q_\mu$ being quasi-convex, by potentially again setting $q_\mu(\alpha)=0$.\qedhere
    \end{enumerate}
\end{proof}
\subsection{Generalized inverses of unimodality-generating non-decreasing functions}
\autoref{thm:unimodaldistributionhasabsolutelycontinuousqunatilefunction} and \autoref{thm: Unimodal distributions via quantile} generalize directly to locally finite \ac{CDF}-unimodal measures.
\begin{theorem}
    \label{thm: absolute continuity of generalized inverse of unimodality-generating non-decreasing function}
    The generalized inverse $\Tilde{H}:I_H\to\R$ of a unimodality-generating non-decreasing function $G:\R\to\R$ is absolutely continuous.
\end{theorem}
\begin{theorem}
    \label{thm: equivalent definition of unimodality-generating via generalized inverse}
    A non-decreasing function $G:\R\to\R$ is unimodality-generating if and only if
    \begin{enumerate}[left=\parindent]
        \item there exists a Radon-Nikodym derivative $h$ of $\Tilde{H}$ which is quasi-convex, or
        \item there exists a mode $\alpha\in[a,b]=\closure{I_H}\subseteq\Rbar$ such that $H$ is concave on $(a,\alpha]$ and convex on $[\alpha,b)$.
    \end{enumerate}
\end{theorem}
Both theorems can be directly proven by replacing $F_\mu$ with $G$ and $Q_\mu$ with $H$ in the proofs of \autoref{thm:unimodaldistributionhasabsolutelycontinuousqunatilefunction} and \autoref{thm: Unimodal distributions via quantile} respectively. As all tools used for these proofs are formulated for general distribution functions, no complication arises. The occurrence of a mode $\nu\in\set{-\infty,\infty}$ possible in this setting even simplifies the proofs because in that case $G$ is absolutely continuous.

\bibliographystyle{amsplain}

\providecommand{\bysame}{\leavevmode\hbox to3em{\hrulefill}\thinspace}
\providecommand{\MR}{\relax\ifhmode\unskip\space\fi MR }
\providecommand{\MRhref}[2]{%
    \href{http://www.ams.org/mathscinet-getitem?mr=#1}{#2}
}
\providecommand{\href}[2]{#2}

\end{document}